\documentclass{amsart}
\usepackage{amssymb}
\usepackage{amsmath}
\usepackage{amsfonts}

\newtheorem{theorem}{Theorem}
\theoremstyle{plain}
\newtheorem{acknowledgement}{Acknowledgement}

\newtheorem{definition}{Definition}
\newtheorem{example}{Example}

\newtheorem{lemma}{Lemma}

\newtheorem{remark}{Remark}

\numberwithin{equation}{section}
\input{tcilatex}

\begin{document}
\title[$\lambda -$statistical convergent function sequences in
intuitionistic fuzzy normed spaces]{$\lambda -$statistical convergent
function sequences in intuitionistic fuzzy normed spaces}
\author{Vatan Karakaya}
\address{Department of Mathematical Engineering, Yildiz Technical
University, Davutpasa Campus, Esenler, 34750 Istanbul}
\email{vkkaya@yildiz.edu.tr;vkkaya@yahoo.com}
\author{Necip Simsek}
\address{Department of Mathematics, Istanbul Commerce University,
Uskudar,Istanbul, Turkey}
\email{necsimsek@yahoo.com; necipsimsek@hotmail.com}
\author{M\"{u}zeyyen Ert\"{u}rk}
\address{Department of Mathematics, Yildiz Technical University, Davutpasa
Campus, Esenler, 34750 Istanbul, Turkey}
\email{merturk@yildiz.edu.tr; merturk3263@gmail.com}
\author{Faik G\"{u}rsoy}
\address{Department of Mathematics, Yildiz Technical University, Davutpasa
Campus, Esenler, 34750 Istanbul, Turkey}
\email{fgursoy@yildiz.edu.tr; faikgursoy02@hotmail.com}
\maketitle

\section{Introduction}

\bigskip Fuzzy logic was introduced by Zadeh \cite{Zadeh} in 1965. Since
then, the importance of fuzzy logic has come increasingly to the
present.There are many applications of fuzzy logic in the field of science
and engineering, e.g. population dynamics \cite{Barros}, chaos control \cite%
{Feng,Fradkov}, computer programming \cite{Giles}, nonlinear dynamical
systems \cite{Hong}, etc. The concept of intuitionistic fuzzy set, as a
generalization of fuzzy logic, was introduced by Atanassov \cite{Atanassov}
in 1986.

Quite recently Park \cite{Park} has introduced the concept of intuitionistic
fuzzy metric space and in \cite{Saadati} , Saadati and Park studied the
notion of intuitionistic fuzzy normed space. Intuitionistic fuzzy analogues
of many concept in classical analysis was studied by many authors \cite%
{mursaleen},\cite{Mursaleen},\cite{Rsaadati},\cite{jebril},\cite{Dinda} etc.

The concept of statistical convergence was introduced by Fast \cite{Fast}.
Mursaleen defined $\lambda -$statistical convergence in\cite{Muhammed}. Also
the concept of statistical convergence was studied in intuitionistic fuzzy
normed space in \cite{Karakus}..Quite recently, Karakaya et al. \cite{vkkaya}
defined and studied statistical convergence of function sequences in
intuitionistic fuzzy normed spaces. Mohiuddine and Lohani defined and
studied $\lambda $-statistical convergence in intuitionistic fuzzy normed
spaces \cite{Lohani}.

In this paper, we shall study concept $\lambda $-statistical convergence for
function sequences and investigate some basic properties related to the
concept in intuitionistic fuzzy normed space.

\begin{acknowledgement}
This work is supported by The Scientific and Technological Research Council
of Turkey (TUBITAK) under the project number 110T699.
\end{acknowledgement}

\begin{definition}
\cite{schw} A binary operation $\ast :\left[ 0,1\right] \times \left[ 0,1%
\right] \rightarrow \left[ 0,1\right] $ is said to be a continuous t-norm if
it satisfies the following conditions:

(i) $\ast $ is associative and commutative ,

(ii) $\ast $ is continuous,

(iii) $a\ast 1=a$ for all $a\in \left[ 0,1\right] $ ,

(iv) \ $a\ast c\leq b\ast d$ whenever $a\leq \ b$\ and $c\leq d$\ for each $%
a,b,c,d\in \left[ 0,1\right] $
\end{definition}

For example, $a\ast b=a.b$\ is a continuous t-norm.

\begin{definition}
\ \cite{schw} A binary operation $\diamond :\left[ 0,1\right] \times \left[
0,1\right] \rightarrow \left[ 0,1\right] $ is said to be a continuous
t-conorm if it satisfies the following conditions:

(i) $\diamond $ is associative and commutative ,

(ii) $\diamond $ is continuous,

(iii) $a\diamond 0=a$ for all $a\in \left[ 0,1\right] $ ,

(iv) \ $a\diamond c\leq b\diamond d$ whenever $a\leq b$\ and $c\leq d$\ for
each $a,b,c,d\in \left[ 0,1\right] $
\end{definition}

For example, $a\diamond b=\min \left\{ a+b,1\right\} $\ is a continuous
t-norm.

\begin{definition}
\cite{Saadati} Let $\ast $\ be a continuous t-norm, $\diamond $ be a
continuous t-conorm and $X$\ be a linear space over the field IF($%
%TCIMACRO{\U{211d} }%
%BeginExpansion
\mathbb{R}
%EndExpansion
$ or $%
%TCIMACRO{\U{2102} }%
%BeginExpansion
\mathbb{C}
%EndExpansion
$) . If $\mu $ and $\nu $\ are fuzzy sets on $X\times \left( 0,\infty
\right) $ satisfying the following conditions, the five-tuple $\left( X,\mu
,\nu ,\ast ,\diamond \right) $ is said to be an intuitionistic fuzzy normed
space and $\left( \mu ,\nu \right) $ is called an intuitionistic fuzzy norm.
For every $x,y\in X$\ and $s,t>0$,

(i) $\mu \left( x,t\right) +\nu \left( x,t\right) \leq 1$ ,

(ii) $\mu \left( x,t\right) >0$,

(iii) $\mu \left( x,t\right) =1\Longleftrightarrow x=0$,

(iv) $\mu \left( ax,t\right) =$ $\mu \left( x,\frac{t}{\left\vert
a\right\vert }\right) $\ for each $a\neq 0$,

(v) $\mu \left( x,t\right) \ast $ $\mu \left( y,s\right) \leq $ $\mu \left(
x+y,t+s\right) $,

(vi) $\mu \left( x,.\right) :\left( 0,\infty \right) \rightarrow \left[ 0,1%
\right] $\ is continuous.

(vii) $\underset{t\rightarrow \infty }{\lim }\mu \left( x,t\right) =1$ and $%
\underset{t\rightarrow 0}{\lim }\mu \left( x,t\right) =0,$

(viii) $\nu \left( x,t\right) <1$,

(ix) $\nu \left( x,t\right) =0\Longleftrightarrow x=0$,

(x) $\nu \left( ax,t\right) =$ $\nu \left( x,\frac{t}{\left\vert
a\right\vert }\right) $\ for each $a\neq 0$,

(xi) $\nu \left( x,t\right) \diamond $ $\nu \left( y,s\right) \geq $ $\nu
\left( x+y,t+s\right) $,

(xii)$\nu \left( x,.\right) :\left( 0,\infty \right) \rightarrow \left[ 0,1%
\right] $is continuous.

(xiii) $\underset{t\rightarrow \infty }{\lim }\nu \left( x,t\right) =1$ and $%
\underset{t\rightarrow 0}{\lim }\nu \left( x,t\right) =0,$
\end{definition}

\textit{For the intuitionistic fuzzy normed space }$\left( X,\mu ,\nu ,\ast
,\diamond \right) ,$\textit{\ as given in \ Dinda and Samanta \cite{Dinda},
we further assume that }$\mu ,\nu ,\ast ,,\diamond $\textit{\ satisfy the
following axiom:}

\textit{(xiv) }$\ \ \ \ \ \left. 
\begin{array}{c}
a\diamond a=a \\ 
a\ast a=a%
\end{array}%
\right\} $\textit{\ for all }$a\in \left[ 0,1\right] .$

\begin{definition}
\cite{Saadati}\textbf{\ }Let $\left( X,\mu ,\nu ,\ast ,\diamond \right) $\
be intuitionistic fuzzy normed space and $\left( x_{k}\right) $\ be sequence
in $X$. $\left( x_{k}\right) $ is said to be convergent to $L\in X$ with
respect to the intuitionistic fuzzy norm $\left( \mu ,\nu \right) $ if for
every $\varepsilon >0$\ and $t>0$, there exists a positive integer $k_{0}$\
such that $\mu \left( x_{k}-L,t\right) >1-\varepsilon $\ and $\nu \left(
x_{k}-L,t\right) <\varepsilon $\ \ whenever $k>k_{0}$. In this case we write 
$\left( \mu ,\nu \right) -\lim x_{k}=L$\ as $k\rightarrow \infty $.
\end{definition}

\begin{definition}
\cite{Saadati} Let $\left( X,\mu ,\nu ,\ast ,\diamond \right) $ be an
intuitionistic fuzzy normed space. For t \TEXTsymbol{>} 0, we define open
ball $B\left( x,r,t\right) $ with center $x\in X$ and radius $0<r<1$, as

\begin{equation*}
B\left( x,r,t\right) =\left\{ y\in X:\mu \left( x-y,t\right) >1-r\text{,}\nu
\left( x-y,t\right) <r\right\}
\end{equation*}
\end{definition}

\begin{definition}
\cite{Dinda} Let $\left( X,\mu ,\nu ,\ast ,\diamond \right) $ and $\left(
Y,\mu ^{\prime },\nu ^{\prime },\ast ,\diamond \right) $\ be two
intuitionistic fuzzy normad linear space over the same field IF. A mapping $%
f $ \ from $\left( X,\mu ,\nu ,\ast ,\diamond \right) $ to $\left( Y,\mu
^{\prime },\nu ^{\prime },\ast ,\diamond \right) $ is said to be
intuitionistic fuzzy continuous at $x_{0}\in X$, if for any given $%
\varepsilon >0,a\in \left( 0,1\right) ,\exists \delta =\delta \left(
a,\varepsilon \right) ,\beta =\beta \left( a,\varepsilon \right) \in \left(
0,1\right) $\ such that for all $x\in X$,%
\begin{equation*}
\mu \left( x-x_{0},\delta \right) >1-\beta \Longrightarrow \mu \prime \left(
f\left( x\right) -f\left( x_{0}\right) ,\varepsilon \right) >1-a
\end{equation*}%
$,$ and $\nu \left( x-x_{0},\delta \right) <\beta \Longrightarrow \nu \prime
\left( f\left( x\right) -f\left( x_{0}\right) ,\varepsilon \right) <a.$
\end{definition}

\begin{definition}
\cite{Dinda} Let $f_{k}:\left( X,\mu ,\nu ,\ast ,\diamond \right)
\rightarrow \left( Y,\mu ^{\prime },\nu ^{\prime },\ast ,\diamond \right) $\
be a sequence of functions. The sequence $\left( f_{k}\right) $ is said to
be pointwise intuitionistic fuzzy convergent on $X$ to a function $f$ with
respect to $\left( \mu ,\nu \right) $ if for each $x\in X$, the sequence $%
\left( f_{k}\left( x\right) \right) \ $is convergent to $f\left( x\right) $
with respect to $\left( \mu \prime ,\nu \prime \right) $.
\end{definition}

\begin{definition}
\cite{Dinda}\textbf{\ }Let $f_{k}:\left( X,\mu ,\nu ,\ast ,\diamond \right)
\rightarrow \left( Y,\mu ^{\prime },\nu ^{\prime },\ast ,\diamond \right) $\
be a sequence of functions. The sequence $\left( f_{k}\right) $ is said to
be uniformly intuitionistic fuzzy convergent on $X$\ to a function $f$ with
respect to $\left( \mu ,\nu \right) $, if given $0<r<1,t>0$, there exist a
positive integer $k_{0}=k_{0}\left( r,t\right) $\ such that $\forall x\in X$
and $\forall k>k_{0}$,%
\begin{equation*}
\mu \prime \left( f_{k}\left( x\right) -f\left( x\right) ,t\right) >1-r,\
\nu \prime \left( f_{k}\left( x\right) -f\left( x\right) ,t\right) <r.
\end{equation*}
\end{definition}

Now, we recall the notion of the statistical convergence of sequences in
intuititonistic fuzzy normed spaces.

\begin{definition}
\cite{freedman}\textbf{\ }Let $K\subset 
%TCIMACRO{\U{2115} }%
%BeginExpansion
\mathbb{N}
%EndExpansion
$\ and $K_{n}=\left\{ k\in K:k\leq n\right\} $.Then the natural density is
defined by $\delta \left( K\right) =\underset{n\rightarrow \infty }{\lim }%
\frac{\left\vert K_{n}\right\vert }{n}$, where $\left\vert K_{n}\right\vert $
denotes the cardinality of \ $K_{n}$.
\end{definition}

\begin{definition}
\cite{Fridy} A sequence $x=\left( x_{k}\right) $ is said to be statistically
convergent to the number $L$ if for every $\varepsilon >0,$ the set $%
N(\varepsilon )$ has asymptotic density zero, where%
\begin{equation*}
N(\varepsilon )=\left\{ k\in 
%TCIMACRO{\U{2115} }%
%BeginExpansion
\mathbb{N}
%EndExpansion
:\left\vert x_{k}-L\right\vert \geq \varepsilon \right\} \text{.}
\end{equation*}%
This case is stated by $st-\lim x=L.$
\end{definition}

\begin{definition}
Let A be subset of $%
%TCIMACRO{\U{2115} }%
%BeginExpansion
\mathbb{N}
%EndExpansion
.$ If a property P(k) holds for all $k\in A$ with $\delta (A)=1$, we say
that P holds for almost all k, that is a.a.k

\begin{definition}
\cite{Karakus}$\ $Let $\left( X,\mu ,\nu ,\ast ,\diamond \right) $ be an
intuitionistic fuzzy normed space. Then, a sequence $\left( x_{k}\right) $
is said to be \ statistically convergent to $L\in X$ with respect to
intuitionistic fuzzy norm \ $\left( \mu ,\nu \right) $ provided that for
every $\varepsilon >0$ and $t>0,$%
\begin{equation*}
\delta \left( \left\{ k\in 
%TCIMACRO{\U{2115} }%
%BeginExpansion
\mathbb{N}
%EndExpansion
:\mu \left( x_{k}-L,t\right) \leq 1-\varepsilon \text{ or }\nu \left(
x_{k}-L,t\right) \geq \varepsilon \right\} \right) =0
\end{equation*}%
or equivalently%
\begin{equation*}
\underset{n\rightarrow \infty }{\lim }\frac{1}{n}\left\vert \left\{ k\leq
n:\mu \left( x_{k}-L,t\right) \leq 1-\varepsilon \text{ or }\nu \left(
x_{k}-L,t\right) \geq \varepsilon \right\} \right\vert =0
\end{equation*}%
This case is stated by $st_{\mu ,\nu }-\lim \left( x_{k}\right) =L$.
\end{definition}

\begin{definition}
\cite{vkkaya}Let $\left( X,\mu ,\nu ,\ast ,\diamond \right) $ and $\left(
Y,\mu ^{\prime },\nu ^{\prime },\ast ,\diamond \right) $ be two
intuitionistic fuzzy normed linear spaces over the same field IF and $\
f_{k}:\left( X,\mu ,\nu ,\ast ,\diamond \right) \rightarrow \left( Y,\mu
^{\prime },\nu ^{\prime },\ast ,\diamond \right) $\ be a sequence of
functions. If for each $x\in X$ and $\forall \varepsilon >0,t>0,$%
\begin{equation*}
\delta \left( \left\{ k\in 
%TCIMACRO{\U{2115} }%
%BeginExpansion
\mathbb{N}
%EndExpansion
:\mu ^{\prime }\left( f_{k}\left( x\right) -f\left( x\right) ,t\right) \leq
1-\varepsilon \text{ or }\nu ^{\prime }\left( f_{k}\left( x\right) -f\left(
x\right) ,t\right) \geq \varepsilon \right\} \right) =0,
\end{equation*}%
then we say hat the sequence $\left( f_{k}\right) $ is pointwise
statistically convergent to $f$\ with respect to intuitionistic fuzzy norm $%
\left( \mu ,\nu \right) $ and we write it $st_{\mu ,\nu }-f_{k}\rightarrow f$%
.
\end{definition}

i.e., for each $x\in X$, $\ \mu ^{\prime }\left( f_{k}\left( x\right)
-f\left( x\right) ,t\right) >1-\varepsilon $ \ and $\ \nu ^{\prime }\left(
f_{k}\left( x\right) -f\left( x\right) ,t\right) <\varepsilon $ $\ \ \ a.a.k$%
.
\end{definition}

\begin{definition}
\cite{vkkaya}Let $\left( X,\mu ,\nu ,\ast ,\diamond \right) $ and $\left(
Y,\mu ^{\prime },\nu ^{\prime },\ast ,\diamond \right) $ be two
intuitionistic fuzzy normed linear space over the same field IF and $\
f_{k}:\left( X,\mu ,\nu ,\ast ,\diamond \right) \rightarrow \left( Y,\mu
^{\prime },\nu ^{\prime },\ast ,\diamond \right) $\ be a sequence of
functions. If for every $x\in X$ and $\forall \varepsilon >0,t>0,$%
\begin{equation*}
\ \delta \left( \left\{ k\in 
%TCIMACRO{\U{2115} }%
%BeginExpansion
\mathbb{N}
%EndExpansion
:\mu ^{\prime }\left( f_{k}\left( x\right) -f\left( x\right) ,t\right) \leq
1-\varepsilon \text{ or }\nu ^{\prime }\left( f_{k}\left( x\right) -f\left(
x\right) ,t\right) \geq \varepsilon \right\} \right) =0,
\end{equation*}%
we say that the sequence $\left( f_{k}\right) $ is uniformly statistically
convergent with respect to $f$\ intuitionistic fuzzy norm $\left( \mu ,\nu
\right) $ and we write it $st_{\mu ,\nu }-f_{k}\rightrightarrows f.$
\end{definition}

\begin{definition}
\cite{Karakus}Let $\left( X,\mu ,\nu ,\ast ,\diamond \right) $ be an
intuitionistic fuzzy normed space. A sequence $\left( x_{k}\right) $ is said
to be statistically Cauchy with respect to intuitionistic fuzzy norm \ $%
\left( \mu ,\nu \right) $ provided that for every $\varepsilon >0$ and $t>0$%
, there exists a number $m\in 
%TCIMACRO{\U{2115} }%
%BeginExpansion
\mathbb{N}
%EndExpansion
$ satisfying%
\begin{equation*}
\delta \left( \left\{ k\in 
%TCIMACRO{\U{2115} }%
%BeginExpansion
\mathbb{N}
%EndExpansion
:\mu \left( x_{k}-x_{m},t\right) \leq 1-\varepsilon \text{ or }\nu \left(
x_{k}-x_{m},t\right) \geq \varepsilon \right\} \right) =0.
\end{equation*}

\begin{definition}
\cite{Muhammed} Let $\lambda =\left( \lambda _{n}\right) $ be a
non-decreasing sequence of positive numbers tending \ to $\infty $ such that%
\begin{equation*}
\lambda _{n+1}\leq \lambda _{n},\text{ \ \ \ \ \ \ }\lambda _{1}=0.
\end{equation*}

Let $K\subset 
%TCIMACRO{\U{2115} }%
%BeginExpansion
\mathbb{N}
%EndExpansion
.$The number 
\begin{equation*}
\delta _{\lambda }\left( K\right) =\underset{n\rightarrow \infty }{\lim }%
\frac{1}{\lambda _{n}}\left\vert \left\{ k\in I_{n}:k\in K\right\}
\right\vert
\end{equation*}

is said to be $\lambda -density$ of $K$, where $I_{n}=\left[ n-\lambda
_{n}+1,n\right] .$

If $\lambda _{n}=n$ for every $n$ then $\lambda -density$ is reduced to
asymtotic density.
\end{definition}

\begin{definition}
\cite{Muhammed} A sequence $x=\left( x_{k}\right) $ is said to be $\lambda -$%
statistically convergent to the number $L$ if for every $\varepsilon >0,$
the set $N(\varepsilon )$ has $\lambda -$density zero, where%
\begin{equation*}
N(\varepsilon )=\left\{ k\in 
%TCIMACRO{\U{2115} }%
%BeginExpansion
\mathbb{N}
%EndExpansion
:\left\vert x_{k}-L\right\vert \geq \varepsilon \right\} \text{.}
\end{equation*}

This case is stated by $st_{\lambda }-\lim x=L.$
\end{definition}
\end{definition}

\begin{definition}
\cite{Lohani}$\ $Let $\left( X,\mu ,\nu ,\ast ,\diamond \right) $ be an
intuitionistic fuzzy normed space. Then, a sequence $\left( x_{k}\right) $
is said to be $\lambda -$\ statistically convergent to $L\in X$ with respect
to intuitionistic fuzzy norm \ $\left( \mu ,\nu \right) $ provided that for
every $\varepsilon >0$ and $t>0,$%
\begin{equation*}
\delta _{\lambda }\left( \left\{ k\in 
%TCIMACRO{\U{2115} }%
%BeginExpansion
\mathbb{N}
%EndExpansion
:\mu \left( x_{k}-L,t\right) \leq 1-\varepsilon \text{ or }\nu \left(
x_{k}-L,t\right) \geq \varepsilon \right\} \right) =0
\end{equation*}%
or equivalently%
\begin{equation*}
\underset{n\rightarrow \infty }{\lim \delta _{\lambda }}\left( \left\{ k\in 
%TCIMACRO{\U{2115} }%
%BeginExpansion
\mathbb{N}
%EndExpansion
:\mu \left( x_{k}-L,t\right) >1-\varepsilon \text{ and }\nu \left(
x_{k}-L,t\right) <\varepsilon \right\} \right) =1
\end{equation*}%
This case is stated by $st_{\mu ,\nu }^{\lambda }-\lim x=L$.
\end{definition}

\section{$\protect\lambda -$Statistical Convergence of Sequence of Functions
in Intuitionistic Fuzzy Normed Spaces}

In this section, we define pointwise $\lambda -$statistical and uniformly $%
\lambda -$statistical convergent sequences of functions in intuitionistic
fuzzy normed spaces. Also, we give the $\lambda -$statistical analog of the
Cauchy convergence criterion for pointwise and uniformly $\lambda -$%
statistical convergent in intuitionistic fuzzy normed space.We investigate
relationship of these concepts with continuity.

\subsection{\textbf{Pointwise }$\protect\lambda -$\textbf{Statistical
Convergence on intuitionistic fuzzy normed spaces }}

\begin{definition}
Let $\left( X,\mu ,\nu ,\ast ,\diamond \right) $ and $\left( Y,\mu ^{\prime
},\nu ^{\prime },\ast ,\diamond \right) $ be two intuitionistic fuzzy normed
linear spaces over the same field IF and $\ f_{k}:\left( X,\mu ,\nu ,\ast
,\diamond \right) \rightarrow \left( Y,\mu ^{\prime },\nu ^{\prime },\ast
,\diamond \right) $\ be a sequence of functions. If for each $x\in X$ and $%
\forall \varepsilon >0,t>0,$%
\begin{equation*}
\delta _{\lambda }\left( \left\{ k\in 
%TCIMACRO{\U{2115} }%
%BeginExpansion
\mathbb{N}
%EndExpansion
:\mu ^{\prime }\left( f_{k}\left( x\right) -f\left( x\right) ,t\right) \leq
1-\varepsilon \text{ or }\nu ^{\prime }\left( f_{k}\left( x\right) -f\left(
x\right) ,t\right) \geq \varepsilon \right\} \right) =0,
\end{equation*}%
or equivalently%
\begin{equation*}
\delta _{\lambda }\left( \left\{ k\in 
%TCIMACRO{\U{2115} }%
%BeginExpansion
\mathbb{N}
%EndExpansion
:\mu ^{\prime }\left( f_{k}\left( x\right) -f\left( x\right) ,t\right)
>1-\varepsilon \text{ and }\nu ^{\prime }\left( f_{k}\left( x\right)
-f\left( x\right) ,t\right) <\varepsilon \right\} \right) =1
\end{equation*}

then we say hat the sequence $\left( f_{k}\right) $ is pointwise $\lambda -$%
statistically convergent with respect to intuitionistic fuzzy norm $\left(
\mu ,\nu \right) $ and we write it $st_{\mu ,\nu }^{\lambda
}-f_{k}\rightarrow f$ .
\end{definition}

\begin{remark}
\item Let $f_{k}:\left( X,\mu ,\nu ,\ast ,\diamond \right) \rightarrow
\left( Y,\mu ^{\prime },\nu ^{\prime },\ast ,\diamond \right) $\ be a
sequence of functions. If $\lambda _{n}=n$ for every $n$ ,since $\lambda
-density$ is reduced to asymtotic density, then$\ \left( f_{k}\right) $ is
pointwise statistically convergent on $X$ with respect to $\left( \mu ,\nu
\right) $ i.e. $st_{\mu ,\nu }-f_{k}\rightarrow f$.
\end{remark}

\begin{lemma}
Let $f_{k}:\left( X,\mu ,\nu ,\ast ,\diamond \right) \rightarrow \left(
Y,\mu ^{\prime },\nu ^{\prime },\ast ,\diamond \right) $\ be a sequence of
functions.Then for every $\varepsilon >0$ and $t>0$, the following
statements are equivalent:
\end{lemma}

(i)$st_{\mu ,\nu }^{\lambda }-f_{k}\rightarrow f$.

(ii)For each $x\in X,$ 
\begin{equation*}
\delta _{\lambda }\left\{ k\in 
%TCIMACRO{\U{2115} }%
%BeginExpansion
\mathbb{N}
%EndExpansion
:\mu ^{\prime }\left( f_{k}\left( x\right) -f\left( x\right) ,t\right) \leq
1-\varepsilon \text{ }\right\} =\delta _{\lambda }\left\{ k\in 
%TCIMACRO{\U{2115} }%
%BeginExpansion
\mathbb{N}
%EndExpansion
:\text{ }\nu ^{\prime }\left( f_{k}\left( x\right) -f\left( x\right)
,t\right) \geq \varepsilon \right\} =0
\end{equation*}

(iii)$\delta _{\lambda }\left\{ k\in 
%TCIMACRO{\U{2115} }%
%BeginExpansion
\mathbb{N}
%EndExpansion
:\mu ^{\prime }\left( f_{k}\left( x\right) -f\left( x\right) ,t\right)
>1-\varepsilon \text{ and }\nu ^{\prime }\left( f_{k}\left( x\right)
-f\left( x\right) ,t\right) <\varepsilon \right\} =1$

(iv)For each $x\in X,$ 
\begin{equation*}
\delta _{\lambda }\left\{ k\in 
%TCIMACRO{\U{2115} }%
%BeginExpansion
\mathbb{N}
%EndExpansion
:\mu ^{\prime }\left( f_{k}\left( x\right) -f\left( x\right) ,t\right)
>1-\varepsilon \text{ }\right\} =\delta _{\lambda }\left\{ k\in 
%TCIMACRO{\U{2115} }%
%BeginExpansion
\mathbb{N}
%EndExpansion
:\text{ }\nu ^{\prime }\left( f_{k}\left( x\right) -f\left( x\right)
,t\right) <\varepsilon \right\} =1
\end{equation*}

(v)For each $x\in X,$ 
\begin{equation*}
st_{\lambda }-lim\mu ^{\prime }(f_{k}\left( x\right) -f\left( x\right) ,t)=1%
\text{ and }st_{\lambda }-lim\nu ^{\prime }(f_{k}\left( x\right) -f\left(
x\right) ,t)=0.
\end{equation*}

$\bigskip $

\begin{example}
Let $\left( 
%TCIMACRO{\U{211d} }%
%BeginExpansion
\mathbb{R}
%EndExpansion
,\left\vert \cdot \right\vert \right) $ denote the space of real numbers
with the usual norm, and let $a\ast b=a.b$ and $a\diamond b=\min \left\{
a+b,1\right\} $ for $a,b\in \left[ 0,1\right] $. For all $x\in 
%TCIMACRO{\U{211d} }%
%BeginExpansion
\mathbb{R}
%EndExpansion
$ and every $t>0,$ consider%
\begin{equation*}
\mu \left( x,t\right) =\frac{t}{t+\left\vert x\right\vert }\text{ and }\nu
\left( x,t\right) =\frac{\left\vert x\right\vert }{t+\left\vert x\right\vert 
}
\end{equation*}%
In this case $\left( 
%TCIMACRO{\U{211d} }%
%BeginExpansion
\mathbb{R}
%EndExpansion
,\mu ,\nu ,\ast ,\diamond \right) $ is intuitionistic fuzzy normed space.
Let $f_{k}:\left[ 0,1\right] \rightarrow 
%TCIMACRO{\U{211d} }%
%BeginExpansion
\mathbb{R}
%EndExpansion
$ be a sequence of functions whose terms are given by%
\begin{equation*}
f_{k}(x)=\left\{ 
\begin{array}{c}
\text{ \ \ \ \ }x^{k}+1\text{, \ \ \ \ \ \ \ \ \ for }0\leq x<\frac{1}{2}%
\text{, if }n-\sqrt{\lambda _{n}}+1\leq k\leq n \\ 
\text{ }0\text{\ ,\ \ \ \ \ \ \ \ \ \ \ \ \ \ \ \ \ \ for }0\leq x<\frac{1}{2%
}\text{, otherwise} \\ 
\text{ \ \ \ \ \ \ \ \ }x^{k}+\frac{1}{2},\text{ \ \ \ \ \ \ \ \ for }\frac{1%
}{2}\leq x<1\text{, if if }n-\sqrt{\lambda _{n}}+1\leq k\leq n \\ 
1,\text{ \ \ \ \ \ \ \ \ \ \ \ \ \ \ \ \ \ \ for }\frac{1}{2}\leq x<1,\text{%
otherwise} \\ 
\text{ }2,\text{ \ \ \ \ \ \ \ \ \ \ \ \ \ \ \ \ \ \ \ \ \ \ \ \ \ \ \ \ \ \
\ \ \ \ \ \ \ \ for }x=1%
\end{array}%
\right. \text{\ \ \ .\ \ \ \ \ \ \ }
\end{equation*}%
$(f_{k})$ is pointwise $\lambda -statistical$ convergent on $\left[ 0,1%
\right] $ with respect to intuitionistic fuzzy norm $(\mu ,\nu ).$ Because,
for $0\leq x<\frac{1}{2}$, $\ $since 
\begin{equation*}
K\left( \varepsilon ,t\right) =\left\{ k\in 
%TCIMACRO{\U{2115} }%
%BeginExpansion
\mathbb{N}
%EndExpansion
:\text{ }\mu \left( f_{k}(x)-f(x),t\right) \leq 1-\varepsilon \text{ or }\nu
\left( f_{k}(x)-f(x),t\right) \geq \varepsilon \right\} ,
\end{equation*}%
hence 
\begin{eqnarray*}
K\left( \varepsilon ,t\right)  &=&\left\{ k\in I_{n}:\frac{t}{t+\left\vert
f_{k}(x)-0\right\vert }\leq 1-\varepsilon \text{ or }\frac{\left\vert
f_{k}(x)-0\right\vert }{t+\left\vert f_{k}(x)-0\right\vert }\geq \varepsilon
\right\}  \\
&=&\left\{ k\in I_{n}:\left\vert f_{k}(x)\right\vert \geq \frac{\varepsilon t%
}{1-\varepsilon }\right\}  \\
&=&\left\{ k\in I_{n}:f_{k}(x)=x^{k}+1\right\} 
\end{eqnarray*}%
and 
\begin{equation*}
\left\vert K\left( \varepsilon ,t\right) \right\vert \leq \sqrt{_{\lambda
_{n}}}
\end{equation*}%
Thus, for $0\leq x<\frac{1}{2}$, since 
\begin{equation*}
\delta _{\lambda }\left( K\left( \varepsilon ,t\right) \right) =\underset{%
n\rightarrow \infty }{\lim }\frac{\left\vert K\left( \varepsilon ,t\right)
\right\vert }{\lambda _{n}}=\underset{n\rightarrow \infty }{\lim }\frac{%
\sqrt{_{\lambda _{n}}}}{\lambda _{n}}=0
\end{equation*}%
$f_{k}$ is $\lambda -statistical$ convergent to $0$ with respect to
intuitionistic fuzzy norm $(\mu ,\nu )$.
\end{example}

For $\frac{1}{2}\leq x<1,$%
\begin{eqnarray*}
K^{\prime }\left( \varepsilon ,t\right) &=&\left\{ k\in I_{n}:\frac{t}{%
t+\left\vert f_{k}(x)-1\right\vert }\leq 1-\varepsilon \text{ or }\frac{%
\left\vert f_{k}(x)-1\right\vert }{t+\left\vert f_{k}(x)-1\right\vert }\geq
\varepsilon \right\} \\
&=&\left\{ k\in I_{n}:\left\vert f_{k}(x)-1\right\vert \geq \frac{%
\varepsilon t}{1-\varepsilon }\right\} \\
&=&\left\{ k\in I_{n}:f_{k}(x)=x^{k}+\frac{1}{2}\right\}
\end{eqnarray*}

\bigskip and 
\begin{equation*}
\left\vert K^{\prime }\left( \varepsilon ,t\right) \right\vert \leq \sqrt{%
_{\lambda _{n}}}
\end{equation*}

Thus, for $0\leq x<\frac{1}{2}$, $f_{k}$ is $\lambda -statistical$
convergent to $1$with respect to intuitionistic fuzzy norm $(\mu ,\nu )$.

For $x=1,$ it can be seen easly that

\begin{eqnarray*}
K^{\prime \prime }\left( \varepsilon ,t\right)  &=&\left\{ k\in 
%TCIMACRO{\U{2115} }%
%BeginExpansion
\mathbb{N}
%EndExpansion
:\frac{t}{t+\left\vert f_{k}(x)-2\right\vert }\leq 1-\varepsilon \text{ or }%
\frac{\left\vert f_{k}(x)-2\right\vert }{t+\left\vert f_{k}(x)-2\right\vert }%
\geq \varepsilon \right\}  \\
&=&\left\{ n\in 
%TCIMACRO{\U{2115} }%
%BeginExpansion
\mathbb{N}
%EndExpansion
:0\geq \frac{\varepsilon t}{1-\varepsilon }\right\}  \\
&=&\oslash 
\end{eqnarray*}

and%
\begin{equation*}
\left\vert K^{\prime \prime }\left( \varepsilon ,t\right) \right\vert =0
\end{equation*}%
and%
\begin{equation*}
\delta _{\lambda }\left( K^{\prime \prime }\left( \varepsilon ,t\right)
\right) =\underset{n\rightarrow \infty }{\lim }\frac{\left\vert K^{\prime
\prime }\left( \varepsilon ,t\right) \right\vert }{\lambda _{n}}=\underset{%
n\rightarrow \infty }{\lim }\frac{0}{\lambda _{n}}=0.
\end{equation*}

Thus for $x=1$, $f_{k}$ is $\lambda -statistical$ convergent to $2$ with
respect to intuitionistic fuzzy norm $(\mu ,\nu )$

\begin{theorem}
Let $\left( X,\mu ,\nu ,\ast ,\diamond \right) $ be an intuitionistic fuzzy
normed space and $f_{k}:\left( X,\mu ,\nu ,\ast ,\diamond \right)
\rightarrow \left( Y,\mu ^{\prime },\nu ^{\prime },\ast ,\diamond \right) $\
be a sequence of functions.If sequence $\left( f_{k}\right) $ is pointwise
intuitionistic fuzzy convergent on $X$ to a function $f$ with respect to $%
\left( \mu ,\nu \right) ,$ then $\left( f_{k}\right) $ is pointwise $\lambda
-$statistical convergent with respect to intuitionistic fuzzy norm $\left(
\mu ,\nu \right) .$
\end{theorem}

\begin{proof}
Let $\forall k\in $ $%
%TCIMACRO{\U{2115} }%
%BeginExpansion
\mathbb{N}
%EndExpansion
$ and $\left( f_{k}\right) $ be pointwise intuitionistic fuzzy convergent in 
$X$. In this case the sequence $\left( f_{k}(x)\right) $ is convergent with
respect to $\left( \mu ^{\prime },\nu ^{\prime }\right) $ for each $x\in X.$
Then for every $\varepsilon >0$ and $t>0,$ there is number $k_{0}\in 
%TCIMACRO{\U{2115} }%
%BeginExpansion
\mathbb{N}
%EndExpansion
$ such that%
\begin{equation*}
\mu ^{\prime }\left( f_{k}\left( x\right) -f\left( x\right) ,t\right)
>1-\varepsilon \text{ and }\nu ^{\prime }\left( f_{k}\left( x\right)
-f\left( x\right) ,t\right) <\varepsilon 
\end{equation*}

for all $\forall k\geq k_{0}$ and for each $x\in X.$Hence for each $x\in X$
the set%
\begin{equation*}
\left\{ k\in 
%TCIMACRO{\U{2115} }%
%BeginExpansion
\mathbb{N}
%EndExpansion
:\mu ^{\prime }\left( f_{k}\left( x\right) -f\left( x\right) ,t\right) \leq
1-\varepsilon \text{ or }\nu ^{\prime }\left( f_{k}\left( x\right) -f\left(
x\right) ,t\right) \geq \varepsilon \right\}
\end{equation*}

has finite numbers of terms. Since finite subset of $%
%TCIMACRO{\U{2115} }%
%BeginExpansion
\mathbb{N}
%EndExpansion
$\ has $\lambda -$density $0$ and hence 
\begin{equation*}
\delta _{\lambda }\left( \left\{ k\in 
%TCIMACRO{\U{2115} }%
%BeginExpansion
\mathbb{N}
%EndExpansion
:\mu ^{\prime }\left( f_{k}\left( x\right) -f\left( x\right) ,t\right) \leq
1-\varepsilon \text{ or }\nu ^{\prime }\left( f_{k}\left( x\right) -f\left(
x\right) ,t\right) \geq \varepsilon \right\} \right) =0.
\end{equation*}

That is, $st_{\mu ,\nu }^{\lambda }-f_{k}\rightarrow f$.
\end{proof}

\begin{theorem}
Let $\left( f_{k}\right) $ and $\left( g_{k}\right) $ be two sequences of
functions from intuitionistic fuzzy normed space $\left( X,\mu ,\nu ,\ast
,\diamond \right) $ to $\left( Y,\mu ^{\prime },\nu ^{\prime },\ast
,\diamond \right) .$ If $st_{\mu ,\nu }^{\lambda }-f_{k}\rightarrow f$ and $%
st_{\mu ,\nu }^{\lambda }-g_{k}\rightarrow g$, then $st_{\mu ,\nu }^{\lambda
}-\left( \alpha f_{k}+\beta g_{k}\right) \rightarrow \alpha f+\beta g$ where 
$\alpha ,\beta \in IF$($%
%TCIMACRO{\U{211d} }%
%BeginExpansion
\mathbb{R}
%EndExpansion
$ or $%
%TCIMACRO{\U{2102} }%
%BeginExpansion
\mathbb{C}
%EndExpansion
$).

\begin{proof}
The proof is clear for $\alpha =0$ and $\beta =0.$Now let $\alpha \neq 0$
and $\beta \neq 0.$ Since $st_{\mu ,\nu }^{\lambda }-f_{k}\rightarrow f$ and 
$st_{\mu ,\nu }^{\lambda }-g_{k}\rightarrow g,$ for each $x\in X$ if we
define 
\begin{equation*}
A_{1}=\left\{ k\in 
%TCIMACRO{\U{2115} }%
%BeginExpansion
\mathbb{N}
%EndExpansion
:\mu ^{\prime }\left( f_{k}\left( x\right) -f\left( x\right) ,\frac{t}{%
2\left\vert \alpha \right\vert }\right) \leq 1-\varepsilon \text{ or }\nu
^{\prime }\left( f_{k}\left( x\right) -f\left( x\right) ,\frac{t}{%
2\left\vert \alpha \right\vert }\right) \geq \varepsilon \right\} 
\end{equation*}%
and 
\begin{equation*}
A_{2}=\left\{ k\in 
%TCIMACRO{\U{2115} }%
%BeginExpansion
\mathbb{N}
%EndExpansion
:\ \mu ^{\prime }\left( g_{k}\left( x\right) -g\left( x\right) ,\frac{t}{%
2\left\vert \beta \right\vert }\right) \leq 1-\varepsilon \text{ or }\nu
^{\prime }\left( g_{k}\left( x\right) -g\left( x\right) ,\frac{t}{%
2\left\vert \beta \right\vert }\right) \geq \varepsilon \right\} 
\end{equation*}%
then $\ \ \ $%
\begin{equation*}
\delta _{\lambda }\left( A_{1}\right) \text{ }=0\ \text{\ }\ \text{and}\
\delta _{\lambda }\left( A_{2}\right) \text{ }=0\ .
\end{equation*}%
$\ \ \ \ \ \ \ \ \ \ \ \ \ \ \ $

$\ $Since $\delta _{\lambda }\left( A_{1}\right) $ $=0\ $\ $\ $and$\ \delta
_{\lambda }\left( A_{2}\right) $ $=0$,\ if we state $A$ by $\left( A_{1}\cup
A_{2}\right) $ then%
\begin{equation*}
\delta _{\lambda }\left( A\right) =0.
\end{equation*}%
Hence $A_{1}\cup A_{2}\neq 
%TCIMACRO{\U{2115} }%
%BeginExpansion
\mathbb{N}
%EndExpansion
$ and there exists $\exists k_{0}\in 
%TCIMACRO{\U{2115} }%
%BeginExpansion
\mathbb{N}
%EndExpansion
$ such that%
\begin{eqnarray*}
\mu ^{\prime }\left( f_{k_{0}}\left( x\right) -f\left( x\right) ,\frac{t}{%
2\left\vert \alpha \right\vert }\right) &>&1-\varepsilon ,\nu ^{\prime
}\left( f_{k_{0}}\left( x\right) -f\left( x\right) ,\frac{t}{2\left\vert
\alpha \right\vert }\right) <\varepsilon , \\
\mu ^{\prime }\left( g_{k_{0}}\left( x\right) -g\left( x\right) ,\frac{t}{%
2\left\vert \beta \right\vert }\right) &>&1-\varepsilon \text{ and }\nu
^{\prime }\left( g_{k_{0}}\left( x\right) -g\left( x\right) ,\frac{t}{%
2\left\vert \beta \right\vert }\right) <\varepsilon
\end{eqnarray*}%
Let%
\begin{eqnarray*}
B &=&\left\{ k\in 
%TCIMACRO{\U{2115} }%
%BeginExpansion
\mathbb{N}
%EndExpansion
:\mu ^{\prime }\left( \left( \alpha f_{k}+\beta g_{k}\right) (x)-\left(
\alpha f(x)+\beta g(x)\right) ,t\right) >1-\varepsilon \text{ and }\right. \\
&&\text{ \ \ \ \ \ \ \ \ \ \ \ \ \ \ \ \ \ \ \ \ \ \ \ \ \ \ \ \ \ \ \ }%
\left. \nu ^{\prime }\left( \left( \alpha f_{k}+\beta g_{k}\right)
(x)-\left( \alpha f(x)+\beta g(x)\right) ,t\right) <\varepsilon \right\} .
\end{eqnarray*}%
We shall show that for each $x\in X$ 
\begin{equation*}
A^{c}\subset B
\end{equation*}%
Let $k_{0}\in A^{c}.$In this case 
\begin{equation*}
\mu ^{\prime }\left( f_{k_{0}}\left( x\right) -f\left( x\right) ,\frac{t}{%
2\left\vert \alpha \right\vert }\right) >1-\varepsilon ,\nu ^{\prime }\left(
f_{k_{0}}\left( x\right) -f\left( x\right) ,\frac{t}{2\left\vert \alpha
\right\vert }\right) <\varepsilon ,
\end{equation*}%
and 
\begin{equation*}
\mu ^{\prime }\left( g_{k_{0}}\left( x\right) -g\left( x\right) ,\frac{t}{%
2\left\vert \beta \right\vert }\right) >1-\varepsilon \text{ and }\nu
^{\prime }\left( g_{k_{0}}\left( x\right) -g\left( x\right) ,\frac{t}{%
2\left\vert \beta \right\vert }\right) <\varepsilon .
\end{equation*}%
Using those above, we have 
\begin{eqnarray*}
\mu ^{\prime }\left( \left( \alpha f_{k_{0}}+\beta g_{k_{0}}\right)
(x)-\left( \alpha f(x)+\beta g(x)\right) ,t\right) &\geq &\mu ^{\prime
}\left( \alpha f_{k_{0}}\left( x\right) -\alpha f(x),\frac{t}{2}\right) \ast
\mu ^{\prime }\left( \beta g_{k_{0}}\left( x\right) -\beta g(x),\frac{t}{2}%
\right) \\
&=&\mu ^{\prime }\left( f_{k_{0}}\left( x\right) -f(x),\frac{t}{2\left\vert
\alpha \right\vert }\right) \ast \mu ^{\prime }\left( g_{k_{0}}\left(
x\right) -g(x),\frac{t}{2\left\vert \beta \right\vert }\right) \\
&>&\left( 1-\varepsilon \right) \ast \left( 1-\varepsilon \right) \\
&=&\left( 1-\varepsilon \right)
\end{eqnarray*}

$\ \ \ \ \ \ \ \ \ \ \ \ \ \ \ \ \ \ \ \ \ \ \ \ \ \ \ \ \ \ \ \ \ \ \ \ \ \
\ \ \ \ \ \ \ \ \ \ \ \ \ \ \ \ \ \ $

$\ $and

\begin{eqnarray*}
\nu ^{\prime }\left( \left( \alpha f_{k_{0}}+\beta g_{k_{0}}\right)
(x)-\left( \alpha f(x)+\beta g(x)\right) ,t\right) &\leq &\nu ^{\prime
}\left( \alpha f_{k_{0}}\left( x\right) -\alpha f(x),\frac{t}{2}\right) \ast
\nu ^{\prime }\left( \beta g_{k_{0}}\left( x\right) -\beta g(x),\frac{t}{2}%
\right) \\
&=&\nu ^{\prime }\left( f_{k_{0}}\left( x\right) -f(x),\frac{t}{2\left\vert
\alpha \right\vert }\right) \ast \nu ^{\prime }\left( g_{k_{0}}\left(
x\right) -g(x),\frac{t}{2\left\vert \beta \right\vert }\right) \\
\ \ \ &<&\varepsilon \diamond \varepsilon \\
&=&\varepsilon
\end{eqnarray*}

\ \ \ \ \ \ \ \ This implies that 
\begin{equation*}
A^{c}\subset B.
\end{equation*}%
Since $B^{c}\subset A$ and $\delta _{\lambda }\left( A\right) $ $=0,$ hence%
\begin{equation*}
\delta _{\lambda }\left( B^{c}\right) =0
\end{equation*}%
That is 
\begin{equation*}
\delta _{\lambda }\left( \left\{ k\in 
%TCIMACRO{\U{2115} }%
%BeginExpansion
\mathbb{N}
%EndExpansion
:\mu ^{\prime }\left( \left( \alpha f_{k}+\beta g_{k}\right) (x)-\left(
\alpha f(x)+\beta g(x)\right) ,t\right) \leq 1-\varepsilon \text{ and }\nu
^{\prime }\left( \left( \alpha f_{k}+\beta g_{k}\right) (x)-\left( \alpha
f(x)+\beta g(x)\right) ,t\right) \geq \varepsilon \right\} \right) =0
\end{equation*}%
$\ \ \ \ \ \ \ \ \ \ \ \ \ \ \ \ \ \ \ \ \ \ \ \ \ \ \ \ \ \ \ \ \ \ \ \ \ \
\ \ \ \ \ \ \ \ \ \ \ \ \ \ \ \ \ \ \ \ \ \ \ \ \ \ \ \ \ \ \ \ \ \ \ $%
\begin{equation*}
st_{\mu ,\nu }^{\lambda }-\left( \alpha f_{k}+\beta g_{k}\right) \rightarrow
\alpha f+\beta g
\end{equation*}%
\ \ \ 
\end{proof}
\end{theorem}

\begin{definition}
\ \ Let $f_{k}:\left( X,\mu ,\nu ,\ast ,\diamond \right) \rightarrow \left(
Y,\mu ^{\prime },\nu ^{\prime },\ast ,\diamond \right) $\ be a sequence of
functions.The sequence $\left( f_{k}\right) $ is a pointwise $\lambda -$%
statistical Cauchy sequence in intuitionistic fuzzy normed space provided
that for every $\varepsilon >0$ and $t>0$ there exists a number $%
N=N(x,\varepsilon ,t)$ such that

\begin{equation*}
\ \delta _{\lambda }\left( \left\{ k\in 
%TCIMACRO{\U{2115} }%
%BeginExpansion
\mathbb{N}
%EndExpansion
:\mu ^{\prime }\left( f_{k}\left( x\right) -f_{N}\left( x\right) ,t\right)
\leq 1-\varepsilon \text{ or }\nu ^{\prime }\left( f_{k}\left( x\right)
-f_{N}\left( x\right) ,t\right) \geq \varepsilon \text{ for each }x\in
X\right\} \right) =0.
\end{equation*}

\begin{theorem}
Let $f_{k}:\left( X,\mu ,\nu ,\ast ,\diamond \right) \rightarrow \left(
Y,\mu ^{\prime },\nu ^{\prime },\ast ,\diamond \right) $\ be a sequence of
functions.If $\left( f_{k}\right) $ is a pointwise $\lambda -$statistical
convergent sequence with respect to intuitionistic fuzzy norm $\left( \mu
,\nu \right) $, then $\left( f_{k}\right) $is a pointwise $\lambda -$%
statistical Cauchy sequence with respect to intuitionistic fuzzy norm $%
\left( \mu ,\nu \right) .$

\begin{proof}
Suppose that $st_{\mu ,\nu }^{\lambda }-f_{k}\rightarrow f$ and let $%
\varepsilon >0$,$t>0$ . For a given $\varepsilon >0,$ choose $s>0$ such that 
$\left( 1-\varepsilon \right) \ast \left( 1-\varepsilon \right) >1-s$ and $%
\varepsilon \diamond \varepsilon <s.$If we state respectively $A_{x}\left(
\varepsilon ,t\right) $ and $A_{x}^{c}\left( \varepsilon ,t\right) $ by 
\begin{equation*}
\left\{ k\in 
%TCIMACRO{\U{2115} }%
%BeginExpansion
\mathbb{N}
%EndExpansion
:\mu ^{\prime }\left( f_{k}\left( x\right) -f\left( x\right) ,\frac{t}{2}%
\right) \leq 1-\varepsilon \text{ or }\nu ^{\prime }\left( f_{k}\left(
x\right) -f\left( x\right) ,\frac{t}{2}\right) \geq \varepsilon \right\} ,
\end{equation*}%
\begin{equation*}
\left\{ k\in 
%TCIMACRO{\U{2115} }%
%BeginExpansion
\mathbb{N}
%EndExpansion
:\mu ^{\prime }\left( f_{k}\left( x\right) -f\left( x\right) ,\frac{t}{2}%
\right) >1-\varepsilon \text{ and }\nu ^{\prime }\left( f_{k}\left( x\right)
-f\left( x\right) ,\frac{t}{2}\right) <\varepsilon \right\} 
\end{equation*}%
for each $x\in X$ .Then, we have 
\begin{equation*}
\delta _{\lambda }\left( A_{x}\left( \varepsilon ,t\right) \right) =0
\end{equation*}

which implies that 
\begin{equation*}
\delta _{\lambda }\left( A_{x}^{c}\left( \varepsilon ,t\right) \right) =1
\end{equation*}%
Let $N\in A_{x}^{c}\left( \varepsilon ,t\right) .$Then%
\begin{equation*}
\mu ^{\prime }\left( f_{N}\left( x\right) -f\left( x\right) ,\frac{t}{2}%
\right) >1-\varepsilon \text{ and }\nu ^{\prime }\left( f_{N}\left( x\right)
-f\left( x\right) ,\frac{t}{2}\right) <\varepsilon 
\end{equation*}%
We want to show that there exists a number $N=N(x,\varepsilon ,t)$ such that 
\begin{equation*}
\ \delta _{\lambda }\left( \left\{ k\in 
%TCIMACRO{\U{2115} }%
%BeginExpansion
\mathbb{N}
%EndExpansion
:\mu ^{\prime }\left( f_{k}\left( x\right) -f_{N}\left( x\right) ,t\right)
\leq 1-s\text{ or }\nu ^{\prime }\left( f_{k}\left( x\right) -f_{N}\left(
x\right) ,t\right) \geq s\text{ for each }x\in X\right\} \right) =0.
\end{equation*}%
Therefore, define for each $x\in X,$%
\begin{equation*}
\ B_{x}\left( \varepsilon ,t\right) =\left\{ k\in 
%TCIMACRO{\U{2115} }%
%BeginExpansion
\mathbb{N}
%EndExpansion
:\mu ^{\prime }\left( f_{k}\left( x\right) -f_{N}\left( x\right) ,t\right)
\leq 1-s\text{ or }\nu ^{\prime }\left( f_{k}\left( x\right) -f_{N}\left(
x\right) ,t\right) \geq s\text{ }\right\} .
\end{equation*}%
We have to show that 
\begin{equation*}
\ B_{x}\left( \varepsilon ,t\right) \subset A_{x}\left( \varepsilon
,t\right) .
\end{equation*}%
Suppose that 
\begin{equation*}
\ B_{x}\left( \varepsilon ,t\right) \nsubseteq A_{x}\left( \varepsilon
,t\right) .
\end{equation*}%
In this case $\ B_{x}\left( \varepsilon ,t\right) $ has at least one
different element which $A_{x}\left( \varepsilon ,t\right) $ doesn't has.
Let $k\in B_{x}\left( \varepsilon ,t\right) \diagdown A_{x}\left(
\varepsilon ,t\right) .$ Then we have 
\begin{equation*}
\mu ^{\prime }\left( f_{k}\left( x\right) -f_{N}\left( x\right) ,t\right)
\leq 1-s\text{ and }\mu ^{\prime }\left( f_{k}\left( x\right) -f\left(
x\right) ,\frac{t}{2}\right) >1-\varepsilon ,
\end{equation*}
in particularly $\mu ^{\prime }\left( f_{N}\left( x\right) -f\left( x\right)
,\frac{t}{2}\right) >1-\varepsilon .$ In this case%
\begin{eqnarray*}
1-s &\geq &\mu ^{\prime }\left( f_{k}\left( x\right) -f_{N}\left( x\right)
,t\right) \geq \mu ^{\prime }\left( f_{k}\left( x\right) -f\left( x\right) ,%
\frac{t}{2}\right) \ast \mu ^{\prime }\left( f_{N}\left( x\right) -f\left(
x\right) ,\frac{t}{2}\right)  \\
\text{ \ \ \ \ \ \ \ \ \ \ \ \ \ \ \ \ \ \ \ \ \ \ \ \ \ \ \ \ \ \ \ \ \ }
&\geq &\left( 1-\varepsilon \right) \ast \left( 1-\varepsilon \right) >1-s,
\end{eqnarray*}%
which is not possible. On the other hand 
\begin{equation*}
\nu ^{\prime }\left( f_{k}\left( x\right) -f_{N}\left( x\right) ,t\right)
\geq s\text{ and }\nu ^{\prime }\left( f_{k}\left( x\right) -f\left(
x\right) ,t\right) <\varepsilon ,
\end{equation*}%
in particularly $\nu ^{\prime }\left( f_{N}\left( x\right) -f\left( x\right)
,t\right) <\varepsilon .$In this case 
\begin{equation*}
\text{ \ \ \ \ \ \ \ \ \ \ \ \ \ \ \ }s\leq \nu ^{\prime }\left( f_{k}\left(
x\right) -f_{N}\left( x\right) ,t\right) \leq \nu ^{\prime }\left(
f_{k}\left( x\right) -f\left( x\right) ,\frac{t}{2}\right) \diamond \nu
^{\prime }\left( f_{N}\left( x\right) -f\left( x\right) ,\frac{t}{2}\right) 
\end{equation*}%
\begin{equation*}
<\varepsilon \diamond \varepsilon <s
\end{equation*}%
which is not possible. Hence $\ B_{x}\left( \varepsilon ,t\right) \subset
A_{x}\left( \varepsilon ,t\right) $. Therefore, by $\delta _{\lambda }\left(
A_{x}\left( \varepsilon ,t\right) \right) =0,$ $\ \delta _{\lambda }\left(
B_{x}\left( \varepsilon ,t\right) \right) =0.$ That is; $\left( f_{k}\right) 
$is a pointwise $\lambda -$statistical Cauchy sequence with respect to
intuitionistic fuzzy norm $\left( \mu ,\nu \right) .$
\end{proof}
\end{theorem}
\end{definition}

\subsection{\textbf{Uniformly }$\protect\lambda -$\textbf{Statistical
Convergence on intuitionistic fuzzy normed spaces}}

\begin{definition}
\textbf{\ }Let $\left( X,\mu ,\nu ,\ast ,\diamond \right) $ and $\left(
Y,\mu ^{\prime },\nu ^{\prime },\ast ,\diamond \right) $ be two
intuitionistic fuzzy normed linear spaces over the same field IF and $\
f_{k}:\left( X,\mu ,\nu ,\ast ,\diamond \right) \rightarrow \left( Y,\mu
^{\prime },\nu ^{\prime },\ast ,\diamond \right) $\ be a sequence of
functions. If for every $x\in X$ and $\forall \varepsilon >0,t>0,$%
\begin{equation*}
\delta _{\lambda }\left( \left\{ k\in 
%TCIMACRO{\U{2115} }%
%BeginExpansion
\mathbb{N}
%EndExpansion
:\mu ^{\prime }\left( f_{k}\left( x\right) -f\left( x\right) ,t\right) \leq
1-\varepsilon \text{ or }\nu ^{\prime }\left( f_{k}\left( x\right) -f\left(
x\right) ,t\right) \geq \varepsilon \right\} \right) =0,
\end{equation*}%
or equivalently%
\begin{equation*}
\delta _{\lambda }\left( \left\{ k\in 
%TCIMACRO{\U{2115} }%
%BeginExpansion
\mathbb{N}
%EndExpansion
:\mu ^{\prime }\left( f_{k}\left( x\right) -f\left( x\right) ,t\right)
>1-\varepsilon \text{ and }\nu ^{\prime }\left( f_{k}\left( x\right)
-f\left( x\right) ,t\right) <\varepsilon \right\} \right) =1
\end{equation*}

then we say hat the sequence $f_{k}$ is uniformly $\lambda -$statistical
convergent with respect to intuitionistic fuzzy norm $\left( \mu ,\nu
\right) $ and we write it $st_{\mu ,\nu }^{\lambda }-f_{k}\rightrightarrows
f $ and
\end{definition}

\begin{remark}
If $st_{\mu ,\nu }^{\lambda }-f_{k}\rightrightarrows f,$ then $st_{\mu ,\nu
}^{\lambda }-f_{k}\rightarrow f.$
\end{remark}

\begin{lemma}
Let $f_{k}:\left( X,\mu ,\nu ,\ast ,\diamond \right) \rightarrow \left(
Y,\mu ^{\prime },\nu ^{\prime },\ast ,\diamond \right) $\ be a sequence of
functions.Then for every $\varepsilon $ \TEXTsymbol{>} 0 and $t>0$, the
following statements are equivalent:
\end{lemma}

(i)$st_{\mu ,\nu }^{\lambda }-f_{k}\rightrightarrows f$.

(ii)For every $x\in X,$ 
\begin{equation*}
\delta _{\lambda }\left\{ k\in 
%TCIMACRO{\U{2115} }%
%BeginExpansion
\mathbb{N}
%EndExpansion
:\mu ^{\prime }\left( f_{k}\left( x\right) -f\left( x\right) ,t\right) \leq
1-\varepsilon \text{ }\right\} =\delta _{\lambda }\left\{ k\in 
%TCIMACRO{\U{2115} }%
%BeginExpansion
\mathbb{N}
%EndExpansion
:\text{ }\nu ^{\prime }\left( f_{k}\left( x\right) -f\left( x\right)
,t\right) \geq \varepsilon \right\} =0
\end{equation*}

(iii)For every $x\in X,$%
\begin{equation*}
\delta _{\lambda }\left\{ k\in 
%TCIMACRO{\U{2115} }%
%BeginExpansion
\mathbb{N}
%EndExpansion
:\mu ^{\prime }\left( f_{k}\left( x\right) -f\left( x\right) ,t\right)
>1-\varepsilon \text{ and }\nu ^{\prime }\left( f_{k}\left( x\right)
-f\left( x\right) ,t\right) <\varepsilon \right\} =1
\end{equation*}

(iv)For every $x\in X,$ 
\begin{equation*}
\delta _{\lambda }\left\{ k\in 
%TCIMACRO{\U{2115} }%
%BeginExpansion
\mathbb{N}
%EndExpansion
:\mu ^{\prime }\left( f_{k}\left( x\right) -f\left( x\right) ,t\right)
>1-\varepsilon \text{ }\right\} =\delta _{\lambda }\left\{ k\in 
%TCIMACRO{\U{2115} }%
%BeginExpansion
\mathbb{N}
%EndExpansion
:\text{ }\nu ^{\prime }\left( f_{k}\left( x\right) -f\left( x\right)
,t\right) <\varepsilon \right\} =1
\end{equation*}

(v)For every $x\in X,$ 
\begin{equation*}
st_{\lambda }-lim\mu ^{\prime }(f_{k}\left( x\right) -f\left( x\right) ,t)=1%
\text{ and }st_{\lambda }-lim\nu ^{\prime }(f_{k}\left( x\right) -f\left(
x\right) ,t)=0.
\end{equation*}

\begin{theorem}
Let $\left( X,\mu ,\nu ,\ast ,\diamond \right) $ be an intuitionistic fuzzy
normed space and $f_{k}:\left( X,\mu ,\nu ,\ast ,\diamond \right)
\rightarrow \left( Y,\mu ^{\prime },\nu ^{\prime },\ast ,\diamond \right) $\
be a sequence of functions.If sequence $\left( f_{k}\right) $ is uniformly
intuitionistic fuzzy convergent on $X$ to a function $f$ with respect to $%
\left( \mu ,\nu \right) ,$ then $\left( f_{k}\right) $ is uniformly $\lambda
-$statistical convergent with respect to intuitionistic fuzzy norm $\left(
\mu ,\nu \right) .$
\end{theorem}

\begin{proof}
Proof of this theorem is similar to proof of theorem that we have done
previously for pointwise $\lambda -statistical$ convergent.
\end{proof}

\begin{theorem}
\bigskip Let $\left( f_{k}\right) $ and $\left( g_{k}\right) $ be two
sequences of functions from intuitionistic fuzzy normed space $\left( X,\mu
,\nu ,\ast ,\diamond \right) $ to$\left( Y,\mu ^{\prime },\nu ^{\prime
},\ast ,\diamond \right) .$ If $st_{\mu ,\nu }^{\lambda
}-f_{k}\rightrightarrows f$ and $st_{\mu ,\nu }^{\lambda
}-g_{k}\rightrightarrows g$, then $st_{\mu ,\nu }^{\lambda }-\left( \alpha
f_{k}+\beta g_{k}\right) \rightrightarrows \alpha f+\beta g$ where $\alpha
,\beta \in IF$($%
%TCIMACRO{\U{211d} }%
%BeginExpansion
\mathbb{R}
%EndExpansion
$ or $%
%TCIMACRO{\U{2102} }%
%BeginExpansion
\mathbb{C}
%EndExpansion
$).
\end{theorem}

\begin{example}
Let $\left( 
%TCIMACRO{\U{211d} }%
%BeginExpansion
\mathbb{R}
%EndExpansion
,\left\vert \cdot \right\vert \right) $ denote the space of real numbers
with the usual norm, and let $a\ast b=a.b$ and $a\diamond b=\min \left\{
a+b,1\right\} $ for $a,b\in \left[ 0,1\right] $. For all $x\in 
%TCIMACRO{\U{211d} }%
%BeginExpansion
\mathbb{R}
%EndExpansion
$ and every $t>0,$ consider%
\begin{equation*}
\mu \left( x,t\right) =\frac{t}{t+\left\vert x\right\vert }\text{ and }\nu
\left( x,t\right) =\frac{\left\vert x\right\vert }{t+\left\vert x\right\vert 
}
\end{equation*}%
Let $f_{k}:\left[ 0,1\right] \rightarrow 
%TCIMACRO{\U{211d} }%
%BeginExpansion
\mathbb{R}
%EndExpansion
$ be a sequence of functions whose terms are given by%
\begin{equation*}
f_{k}(x)=\left\{ 
\begin{array}{c}
\text{ \ \ \ \ }x^{k}+1\text{, \ \ \ \ \ \ \ \ \ if }n-\sqrt{\lambda _{n}}%
+1\leq k\leq n \\ 
\text{ }0\text{\ ,\ \ \ \ \ \ \ \ \ \ \ \ \ \ \ \ \ otherwise}%
\end{array}%
\right. \text{\ \ \ .\ \ \ \ \ \ \ }
\end{equation*}%
since $K\left( \varepsilon ,t\right) =\left\{ k\in I_{n}:\text{ }\mu \left(
f_{k}(x)-f(x),t\right) \leq 1-\varepsilon \text{ or }\nu \left(
f_{k}(x)-f(x),t\right) \geq \varepsilon \right\} $ ,hence 
\begin{eqnarray*}
K\left( \varepsilon ,t\right)  &=&\left\{ k\in I_{n}:\frac{t}{t+\left\vert
f_{k}(x)-0\right\vert }\leq 1-\varepsilon \text{ or }\frac{\left\vert
f_{k}(x)-0\right\vert }{t+\left\vert f_{k}(x)-0\right\vert }\geq \varepsilon
\right\}  \\
&=&\left\{ k\in I_{n}:\left\vert f_{k}(x)\right\vert \geq \frac{\varepsilon t%
}{1-\varepsilon }\right\}  \\
&=&\left\{ k\in I_{n}:f_{k}(x)=x^{k}+1\right\} 
\end{eqnarray*}%
and 
\begin{equation*}
\left\vert K\left( \varepsilon ,t\right) \right\vert \leq \sqrt{_{\lambda
_{n}}}
\end{equation*}%
Thus, for $0\leq x\leq 1$, since 
\begin{equation*}
\delta _{\lambda }\left( K\left( \varepsilon ,t\right) \right) =\underset{%
n\rightarrow \infty }{\lim }\frac{\left\vert K\left( \varepsilon ,t\right)
\right\vert }{\lambda _{n}}=\underset{n\rightarrow \infty }{\lim }\frac{%
\sqrt{_{\lambda _{n}}}}{\lambda _{n}}=0
\end{equation*}%
$f_{k}$ is uniform $\lambda -statistical$ convergent to $0$ with respect to
intuitionistic fuzzy norm $(\mu ,\nu )$.
\end{example}

\begin{definition}
\bigskip\ \ Let $f_{k}:\left( X,\mu ,\nu ,\ast ,\diamond \right) \rightarrow
\left( Y,\mu ^{\prime },\nu ^{\prime },\ast ,\diamond \right) $\ be a
sequence of functions.The sequence $\left( f_{k}\right) $ is a uniformly $%
\lambda -$statistical Cauchy sequence in intuitionistic fuzzy normed space
provided that for every $\varepsilon >0$ and $t>0$ there exists a number $%
N=N(\varepsilon ,t)$ such that

\begin{equation*}
\ \delta _{\lambda }\left( \left\{ k\in 
%TCIMACRO{\U{2115} }%
%BeginExpansion
\mathbb{N}
%EndExpansion
:\mu ^{\prime }\left( f_{k}\left( x\right) -f_{N}\left( x\right) ,t\right)
\leq 1-\varepsilon \text{ or }\nu ^{\prime }\left( f_{k}\left( x\right)
-f_{N}\left( x\right) ,t\right) \geq \varepsilon \text{ for every }x\in
X\right\} \right) =0.
\end{equation*}

\begin{theorem}
Let $f_{k}:\left( X,\mu ,\nu ,\ast ,\diamond \right) \rightarrow \left(
Y,\mu ^{\prime },\nu ^{\prime },\ast ,\diamond \right) $\ be a sequence of
functions.If $\left( f_{k}\right) $ is a uniformly $\lambda -$statistically
convergent sequence with respect to intuitionistic fuzzy norm $\left( \mu
,\nu \right) $, then $\left( f_{k}\right) $is a uniformly $\lambda -$%
statistical Cauchy sequence with respect to intuitionistic fuzzy norm $%
\left( \mu ,\nu \right) .$
\end{theorem}
\end{definition}

\begin{proof}
Proof of this theorem is similar to proof of theorem that we have done
previously for pointwise $\lambda -statistical$ convergent
\end{proof}

\begin{definition}
\bigskip Let $\left( X,\mu ,\nu ,\ast ,\diamond \right) $ and $\left( Y,\mu
^{\prime },\nu ^{\prime },\ast ,\diamond \right) $ be two intuitionistic
fuzzy normed space, and $F$ a family of functions from $X$ to $Y$. \ The
family $F$ is intuitionistic fuzzy equicontinuous at a point $x_{0}\in X$ if
for every $\varepsilon >0$ and $t>0$, there exists a $\delta >0$ such that $%
\mu ^{\prime }\left( f\left( x_{0}\right) -f\left( x\right) ,t\right)
>1-\varepsilon $ and $\nu ^{\prime }\left( f\left( x_{0}\right) -f\left(
x\right) ,t\right) <\varepsilon $ for all $f\in F$ and all $x$ such that$\mu
^{\prime }\left( x_{0}-x,t\right) >1-\delta $ and $\nu ^{\prime }\left(
x_{0}-x,t\right) <\delta $. The family is intuitionistic fuzzy
equicontinuous if it is equicontinuous at each point of $X.($For continuity, 
$\delta $ may depend on $\varepsilon $, $x_{0}$ and $f$; for equicontinuity, 
$\delta $ must be independent of $f$)
\end{definition}

\begin{theorem}
\bigskip let $\left( X,\mu ,\nu ,\ast ,\diamond \right) ,\left( Y,\mu
^{\prime },\nu ^{\prime },\ast ,\diamond \right) $ be intuitionistic fuzzy
normed space.Assume that $st_{\mu ,\nu }^{\lambda }-f_{k}\rightarrow f$ \ on 
$X$ where functions $f_{k}:\left( X,\mu ,\nu ,\ast ,\diamond \right)
\rightarrow \left( Y,\mu ^{\prime },\nu ^{\prime },\ast ,\diamond \right)
,k\in 
%TCIMACRO{\U{2115} }%
%BeginExpansion
\mathbb{N}
%EndExpansion
,$ are intuitionistic fuzzy equi-continuous on $X$ and $f:X\rightarrow Y.$
Then $f$ is continuous on $X.$
\end{theorem}

\begin{proof}
Let $x_{0}\in X$ be an arbitrary point. By the intuitionistic fuzzy
equi-continuity of $f_{k}$'s,for every $\varepsilon >0$ and $t>0$ there
exist $\delta =\delta \left( x_{0},\varepsilon ,\frac{t}{3}\right) >0$ such
that 
\begin{equation*}
\mu ^{\prime }\left( f_{k}\left( x_{0}\right) -f_{k}\left( x\right) ,\frac{t%
}{3}\right) >1-\varepsilon \text{ and }\nu ^{\prime }\left( f_{k}\left(
x_{0}\right) -f_{k}\left( x\right) ,\frac{t}{3}\right) <\varepsilon 
\end{equation*}%
for every $k\in 
%TCIMACRO{\U{2115} }%
%BeginExpansion
\mathbb{N}
%EndExpansion
$ and all $x$ such that $\mu ^{\prime }\left( x_{0}-x,t\right) >1-\delta $
and $\nu ^{\prime }\left( x_{0}-x,t\right) <\delta $. Let $x\in B\left(
x_{0},\delta ,t\right) $ be fixed.Since $st_{\mu ,\nu }^{\lambda
}-f_{k}\rightarrow f$ \ on $X,$for each $x\in X$, if we state respectively $%
A_{x}\left( \varepsilon ,t\right) $ and $B$ $_{x}\left( \varepsilon
,t\right) $ by the sets 
\begin{equation*}
A_{x}\left( \varepsilon ,t\right) =\left\{ k\in 
%TCIMACRO{\U{2115} }%
%BeginExpansion
\mathbb{N}
%EndExpansion
:\mu ^{\prime }\left( f_{k}\left( x\right) -f\left( x\right) ,\frac{t}{3}%
\right) \leq 1-\varepsilon \text{ or }\nu ^{\prime }\left( f_{k}\left(
x\right) -f\left( x\right) ,\frac{t}{3}\right) \geq \varepsilon \text{ for
each }x\in X\right\} 
\end{equation*}

and 
\begin{equation*}
B_{x}\left( \varepsilon ,t\right) =\left\{ k\in 
%TCIMACRO{\U{2115} }%
%BeginExpansion
\mathbb{N}
%EndExpansion
:\mu ^{\prime }\left( f_{k}\left( x_{0}\right) -f\left( x_{0}\right) ,\frac{t%
}{3}\right) \leq 1-\varepsilon \text{ or }\nu ^{\prime }\left( f_{k}\left(
x_{0}\right) -f\left( x_{0}\right) ,\frac{t}{3}\right) \geq \varepsilon 
\text{ for each }x\in X\right\} 
\end{equation*}%
then, $\delta _{\lambda }\left( A_{x}\left( \varepsilon ,t\right) \right) =0$
and $\delta _{\lambda }\left( B_{x}\left( \varepsilon ,t\right) \right) =0,$
hence$\delta _{\lambda }\left( A_{x}\left( \varepsilon ,t\right) \cup
B_{x}\left( \varepsilon ,t\right) \right) =0$ and $A_{x}\left( \varepsilon
,t\right) \cup B_{x}\left( \varepsilon ,t\right) $ is different fom $%
%TCIMACRO{\U{2115} }%
%BeginExpansion
\mathbb{N}
%EndExpansion
.$ Thus, there exists $\exists m\in 
%TCIMACRO{\U{2115} }%
%BeginExpansion
\mathbb{N}
%EndExpansion
$ such that 
\begin{equation*}
\mu ^{\prime }\left( f_{m}\left( x\right) -f\left( x\right) ,\frac{t}{3}%
\right) >1-\varepsilon \ ,\text{ }\nu ^{\prime }\left( f_{m}\left( x\right)
-f\left( x\right) ,\frac{t}{3}\right) <\varepsilon 
\end{equation*}

and%
\begin{equation*}
\mu ^{\prime }\left( f_{m}\left( x_{0}\right) -f\left( x_{0}\right) ,\frac{t%
}{3}\right) >1-\varepsilon \ ,\text{ }\nu ^{\prime }\left( f_{m}\left(
x_{0}\right) -f\left( x_{0}\right) ,\frac{t}{3}\right) <\varepsilon \text{.}
\end{equation*}

Now, we will show that $f$ is intuitionistic fuzzy contunious at $x_{0}.$%
Since $st_{\mu ,\nu }^{\lambda }-f_{k}\rightarrow f$ \ and  for every $k\in 
%TCIMACRO{\U{2115} }%
%BeginExpansion
\mathbb{N}
%EndExpansion
$ $\ f_{k}$'s\ \ are continuous, $f_{m}$\ is also continuous for $m\in 
%TCIMACRO{\U{2115} }%
%BeginExpansion
\mathbb{N}
%EndExpansion
$, we have

\begin{eqnarray*}
\mu \prime \left( f\left( x\right) -f\left( x_{0}\right) ,t\right)  &=&\mu
\prime \left( f\left( x\right)
-f_{m}(x)+f_{m}(x)-f_{m}(x_{0})+f_{m}(x_{0})-f\left( x_{0}\right) ,t\right) 
\\
&\geq &\mu \prime \left( f\left( x\right) -f_{m}(x),\frac{t}{3}\right) \ast
\mu \prime \left( f_{m}(x)-f_{k}(x_{0}),\frac{t}{3}\right) \ast \mu \prime
\left( f_{m}(x_{0})-f\left( x_{0}\right) ,\frac{t}{3}\right)  \\
&>&1-\varepsilon \ast 1-\varepsilon \ast 1-\varepsilon  \\
&=&1-\varepsilon 
\end{eqnarray*}

$\ $and

\begin{eqnarray*}
\nu \prime \left( f\left( x\right) -f\left( x_{0}\right) ,t\right)  &=&\nu
\prime \left( f\left( x\right)
-f_{m}(x)+f_{m}(x)-f_{m}(x_{0})+f_{m}(x_{0})-f\left( x_{0}\right) ,t\right) 
\\
&\leq &\nu \prime \left( f\left( x\right) -f_{m}(x),\frac{t}{3}\right) \ast
\nu \prime \left( f_{m}(x)-f_{m}(x_{0}),\frac{t}{3}\right) \ast \nu \prime
\left( f_{m}(x_{0})-f\left( x_{0}\right) ,\frac{t}{3}\right)  \\
&<&\varepsilon \diamond \varepsilon \diamond \varepsilon  \\
&=&\varepsilon .
\end{eqnarray*}

Thus,\ the proof is completed.
\end{proof}

\end{document}